\documentclass[11pt]{amsart}

\usepackage{amssymb}
\usepackage[all]{xy}
\usepackage[colorlinks=true, urlcolor=rltblue, citecolor=drkgreen, linkcolor=drkred] {hyperref}
\usepackage{color}
\usepackage{pdfsync}
\definecolor{rltblue}{rgb}{0,0,0.4}
\definecolor{drkred}{rgb}{0.6,0,0}
\definecolor{drkgreen}{rgb}{0,0.4,0}
\usepackage{graphicx}
\usepackage[mathscr]{eucal}

\usepackage{lmodern}
\usepackage[capitalize]{cleveref}
\usepackage{amssymb,amsmath,amsthm}
\usepackage{mathtools, MnSymbol}
\usepackage{setspace}
\usepackage{tikz-cd}
\usepackage[T1]{fontenc}
\usepackage[utf8]{inputenc}
\usepackage{textcomp} 
\usepackage{stmaryrd}
\usepackage{datetime}
\usepackage{todonotes}
\usepackage{xcolor}
\usepackage{mdframed}
\usepackage[all]{xy}

\usepackage{thmtools}
\usepackage{enumitem}
\declaretheorem{theorem}
\declaretheorem[sibling=theorem]{lemma}
\declaretheorem[sibling=theorem]{proposition}
\declaretheorem[sibling=theorem]{corollary}
\declaretheorem[sibling=theorem]{definition}

\declaretheorem[numberwithin=theorem]{claim}

\newcommand{\A}{\mathcal{A}}
\newcommand{\B}{\mathcal{B}}
\newcommand{\C}{\mathcal{C}}

\renewcommand{\phi}{\varphi}
\newcommand{\bigwwedge}{%
  \mathop{
    \mathchoice{\bigwedge\mkern-15mu\bigwedge}
               {\bigwedge\mkern-12.5mu\bigwedge}
               {\bigwedge\mkern-12.5mu\bigwedge}
               {\bigwedge\mkern-11mu\bigwedge}
    }
}

\newmdtheoremenv[backgroundcolor=cyan]{theorem-prove}{Theorem}[theorem]
\newmdtheoremenv[backgroundcolor=cyan]{lemma-prove}{Lemma}[theorem]
\newmdtheoremenv[backgroundcolor=cyan]{proposition-prove}{Proposition}[theorem]

\newmdtheoremenv[backgroundcolor=yellow!40]{theorem-check}{Theorem}[theorem]
\newmdtheoremenv[backgroundcolor=yellow!40]{lemma-check}{Lemma}[theorem]
\newmdtheoremenv[backgroundcolor=yellow!40]{proposition-check}{Proposition}[theorem]

\def\hbar{{\bar{h}}}

\def\A{\mathcal A}
\def\B{\mathcal{B}}
\def\C{\mathcal{C}}

\def\H{\mathcal H}

\newtheorem{thm}{Theorem}

\theoremstyle{remark}

\newtheorem{example}[thm]{Example}

\def\and{\mathrel{\&}}

\def\RK{\textbf{RK}}

\title{Computability of Separation Axioms in Countable Second Countable Spaces}
\author{Andrew DeLapo}
\author{David Gonzalez}

\address[DeLapo]{University of Connecticut\\
Department of Mathematics\\
341 Mansfield Road, Storrs, CT 06269\\
  USA}
\email{\href{andrew.delapo@uconn.edu}{andrew.delapo@uconn.edu}}
\urladdr{\url{https://adelapo.github.io}}

\address[Gonzalez]{University of Notre Dame\\
Department of Mathematics\\
Hurley Hall, 255 Hurley, Notre Dame, IN 46556\\
  USA}
\email{\href{dgonza42@nd.edu}{dgonza42@nd.edu}}
\urladdr{\url{https://www.davidgonzalezlogic.com}}

\begin{document}
\maketitle

\begin{abstract}
    We analyze the effective content of countable, second countable topological spaces by directly calculating the complexity of several topologically defined index sets.
    We focus on the separation principles, calibrating an arithmetic completeness result for each of the Tychonoff separation axioms.
    Beyond this, we prove completeness results for various other topological properties, such as being Polish and having a particular Cantor-Bendixson rank, using tools from computable structure theory. 
    This work contrasts with previous work analyzing countable, second countable spaces which used the framework of reverse mathematics, as reverse mathematics generally lacks the precision to pin down exact arithmetic complexity levels for properties of interest.
\end{abstract}

\section{Introduction}

In the framework of computability theory, one can rigorously define the complexity of a subset of the natural numbers. The definable sets of natural numbers form a hierarchy, and this allows for the complexities of sets to be compared to established benchmarks as well as other sets. If we fix an encoding of a certain class of structures by natural numbers and then consider the subset of (codes of) such structures satisfying a desired property, then we arrive at a characterization of the complexity of that structural property. This analysis has been carried out for certain classes of algebraic structures (for example, see \cite{C04, C05, KS17}) as well as some topological structures (see \cite{HTM23, T23}). In this article, we classify several point-set topological properties in the context of countable, second countable (CSC) topological spaces, in particular the separation axioms $T_0$, $T_1$, $T_2$, $T_3$, and others.

The Tychonoff $T_i$ hierarchy of properties is a fundamental subject of study in point set topology (\cite[Chapter 4]{K75} is a standard reference).
These properties were first studied well over one hundred years ago and are tied with the early development of topology.
To this day, they serve as important benchmarks measuring the extent to which topological spaces resemble metric spaces.
Because of their central place in general topology, they are the natural initial subjects for our direct calculations of index sets of CSC spaces.

Within computability theory, CSC spaces were first considered by Dorais in \cite{D11}. Compared to other formalizations of topological spaces within computability, such as Polish spaces or maximal filter spaces, working with CSC spaces has two significant advantages. First, points and (indices for) basic open sets are coded directly by natural numbers, so there is reduced overhead in writing definitions using the language of arithmetic. Second, CSC spaces can exhibit a wider array of general properties than other formalizations of topological spaces; for example, Polish spaces are always $T_3$ (in fact, always completely metrizable), and maximal filter spaces are always $T_1$ and are completely metrizable when they are $T_3$. On the other hand, a CSC space need not even be $T_0$ and may exhibit broadly general behavior.
The CSC formalization is the only way to directly study the complexity of the separation axioms and related notions from general topology, as other formalizations are simply not general enough.

CSC spaces have been the formalization of choice in some recent work in reverse mathematical analysis of topological principles for similar reasons; a particular theorem in combinatorial topology was restricted to CSC spaces and analyzed by Benham, DeLapo, Dzhafarov, Solomon, and Villano in \cite{BDDSV24}. Dorais' 2011 paper \cite{D11} compared notions of compactness for CSC spaces. Shafer \cite{S20} considered compactness specifically for the order topology on countable complete linear orders; the order topology on a countable linear order is a CSC space, and in fact is always $T_3$. Genovesi \cite{G24} also studied several aspects of $T_3$ CSC spaces. In the second half of this paper, we apply properties of certain linear orders to characterize a variety of classes of $T_3$ CSC spaces, thereby contributing to this line of inquiry. 

We define ``CSC space'' here as originally given by Dorais in \cite{D11}.
While Dorais's definition was over the subsystem $\mathsf{RCA_0}$ of second-order arithmetic, our investigation is not within reverse mathematics, so we dispense with this formalism.

\begin{definition}
    A \textbf{base} for a (countable, second-countable) topology on a countable set $X$ is a sequence $\mathcal{U} = (U_i)_{i \in \omega}$ of subsets of $X$, together with a function $k: \omega \times \omega \times X \to \omega$, such that
    \begin{itemize}
        \item for all $x \in X$, there is $i \in \omega$ such that $x \in U_i$, and
        \item for all $x \in X$ and $i, j \in \omega$, if $x \in U_i \cap U_j$, then $x \in U_{k(i, j, x)}$.
    \end{itemize}
\end{definition}

\begin{definition}
    A \textbf{countable second-countable (CSC) space} is a triple $(X, \mathcal{U}, k)$ where $\mathcal{U}$ and $k$ form a base for a countable, second-countable topology on $X$.
\end{definition}

Classically, if $X$ is a topological space with a basis $\mathcal{U}$ of basic open sets, then there is always a function $k$ with the above property. When CSC spaces are studied within reverse mathematics, requiring the $k$ function to exist places some restriction on the types of CSC spaces that can be formed. We will see that the $k$ function has less of a prominent role in our discussion, but we will continue to include $k$ in the signature of a CSC space. In particular, if a CSC space $(X, \mathcal{U}, k)$ is to be considered ``computable,'' then we should expect $k$ to be a computable function.
We formalize the notion of a computable CSC space in the subsequent section, alongside the notion of index complexity within a larger class of spaces in the style of \cite{WC05}.
Using these definitions, we prove a series of theorems, classifying the complexities of the index sets for each of the $T_i$ principles in the Tychonoff hierarchy.
These are summarized in the theorem below.
\begin{theorem}
    \begin{enumerate}
        \item The $T_0$ spaces are $\Pi_2^0$ complete within the CSC spaces.
        \item The $T_{1/2}$ spaces are $\Pi_4^0$ complete within the $T_0$ spaces.
        \item The $T_1$ spaces are $\Pi_2^0$ complete within the $T_{1/2}$ spaces.
        \item The $T_2$ spaces are $\Pi_3^0$ complete within the $T_1$ spaces.
        \item The $T_{2\frac12}$ spaces are $\Pi_5^0$ complete within the $T_2$ spaces.
        \item The $T_3$ or metrizable spaces are $\Pi_5^0$ complete within the $T_{2\frac12}$ spaces.
    \end{enumerate}
\end{theorem}
\noindent These results are of particular interest because they pierce quite deeply into the arithmetic hierarchy.
It is not common to see a property as natural as metrizability in a topological space as far as the $\Pi_5^0$ level of the arithmetic hierarchy.
These natural index set calculations also represent a new approach in the study of CSC spaces and computable, general topology.

We move beyond the classical hierarchy of separation principles to complete metrizability or Polishness.
The approach to this problem uses different techniques.
In particular, by systematically turning linear orderings into metrizable topological spaces in the style of \cite{S20,G24}, we can leverage powerful tools from computable structure theory, like the pair of structures theorem \cite{AK90}, to simplify our index set calculations.
We obtain the desired calibration of complete metrizability.
\begin{theorem}
    The completely metrizable spaces are $\Pi_1^1$ complete within the metrizable spaces.
\end{theorem}
\noindent It should be noted that this theorem statement bears similarities to \cite[Corollary 12.10]{G24}, but is outside of the context of reverse math and is approached in an entirely different manner.
Our approach to this problem has the advantage that we obtain several other index set results essentially ``for free'' on the way to completing our proof.
We highlight an illustrative example below.
\begin{proposition}
    The spaces with Cantor-Bendixson rank $\alpha$ are $\Pi_{2\alpha+3}^0$ complete within the metrizable spaces.
\end{proposition}

Our article is organized into three sections, including the current introductory section.
Section 2 concerns index set calculations for the classical hierarchy of separation principles up to metrizability.
Section 3 moves beyond metrizability and uses techniques from computable structure theory to calculate the index set for the completely metrizable spaces and related notions.

\section{Separation Principles Up to Metrizability}
In this section, we calculate the optimal complexity for the index sets of countable second countable spaces with separation properties up to the metric level.
The section is organized into three subsections.
In the first subsection, we define the fundamental notions of index set and relative complexity needed throughout the article.
In the second subsection, we analyze the most standard separation hierarchy of $T_0,T_1,T_2,$ and $T_3$.
Note that because our spaces are second countable, the Urysohn metrization theorem guarantees that $T_3$ spaces are the same as metrizable spaces, and therefore are the same as $T_4$ spaces (or any stronger property implied by metrizability).
In the final subsection, we analyze the important intermediate notions of $T_{\frac12}$ and $T_{2 \frac12}$.

\subsection{Definitions for important terms}
To be precise, we take the following definition for the index of CSC space.
\begin{definition}
    An \textbf{index} for a CSC space is a pair $\left\langle m, n \right\rangle$ such that
    \begin{itemize}
        \item $\Phi_m$ is the characteristic function for $\mathcal{U} = (U_i)_{i \in \omega}$, meaning $\Phi_m$ is total, and for all $i, x \in \omega$, $\Phi_m(i, x) = 1$ if and only if $x \in U_i$, and $\Phi_m(i, x) = 0$ if and only if $x \notin U_i$,
        \item $\Phi_n$ calculates the $k$ function for the space; in other words, $\Phi_n$ is total and for all $i,j,x\in\omega$ with $x\in U_i\cap U_j$, $\Phi_m(\Phi_n( i,j,x),x)=1$ and for all $y$ with $\Phi_m(\Phi_n(i,j,x),y)=1$ both $\Phi_m(i,y)=1$ and $\Phi_m(j,y)=1$.
    \end{itemize}

    Write $CSC = \{e \in \omega : \text{$e$ is an index for a CSC space}\}$.
\end{definition}

It should be noted that being an index for a CSC space is already a non-trivial property.
\begin{theorem}
    The set $CSC$ is $\Pi^0_2$-complete.
\end{theorem}

\begin{proof}
    It is immediate from the definition outlined above that $CSC$ is $\Pi^0_2$.
    Specifically, ensuring the totality of $\Phi_m$ and $\Phi_n$ are $\Pi^0_2$ properties.

    To see that the set is $\Pi^0_2$-complete, we reduce the set of total functions to $CSC$.
    Given $\Phi_e$ we define $\Phi_{m(e)}$ and $\Phi_{n(e)}$ on the underlying space $\omega$.
    Let
    \[\Phi_{m(e)}(i,x)=
    \begin{cases}
        0 \text{ if } i\neq x \\
        1 \text{ if } i=x \text{ and } \Phi_e(x)\downarrow \\
        \uparrow \text{ if } i=x \text{ and } \Phi_e(x)\uparrow
    \end{cases},\]
    and
    \[ \Phi_{n(e)}(i,j,x)=i.\]
    If $\Phi_e$ is not total, it is apparent that $\Phi_{m(e)}$ is not total either, so $\langle m(e),n(e)\rangle$ is not in $CSC$.
    If $\Phi_e$ is total, $\Phi_{m(e)}$ puts $x\in U_i$ precisely when $x=i$; it is straightforward to confirm that $\Phi_{n(e)}$ acts as required in this case, so $\langle m(e),n(e)\rangle$ is in $CSC$.
\end{proof}

Because there is some complexity inherent in ensuring that a CSC space is coded, we take efforts to avoid relying on this complexity when constructing our index set examples.
More generally, when it comes to separation principles, we have that, for example, $T_1\subset T_0$.
We also need to make similar efforts to ensure that, say, the coding power of $T_1$ is not entirely dependent on $T_1$ spaces already being $T_0$ spaces.
We formalize this using the following notion due to Calvert and Knight \cite{WC05}.

\begin{definition}
    Let $\Gamma$ be a complexity class and $A \subseteq C$.
    \begin{itemize}
        \item $A$ is $\Gamma$ \textbf{within} $C$ if $A = B \cap C$ for some $B \in \Gamma$.
        \item $A$ is \textbf{$\Gamma$-hard within} $C$ if for every $B \in \Gamma$, there is a computable function $f: \omega \to \omega$ such that for all $e \in \omega$,
        \begin{equation*}
            \text{$f(e) \in C$ and $(e \in B \iff f(e) \in A)$}.
        \end{equation*}
        \item $A$ is \textbf{$\Gamma$-complete within} $C$ if $A$ is $\Gamma$ within $C$ and $A$ is $\Gamma$-hard within $C$.
    \end{itemize}
\end{definition}

Suppose we have sets $A\subseteq C \subseteq CSC$, which are a set of indices of CSC spaces having some topological properties. We want to show $A$ is $\Gamma$-hard within $C$ for some complexity class $\Gamma$. A typical proof will proceed as follows:
\begin{enumerate}
    \item Fix a set $B \in \Gamma$. Usually, $B$ is a known $\Gamma$-complete set.
    \item Let $e \in \omega$. Define a sequence of sets $\mathcal{V} = (V_i)_{i \in \omega}$ uniformly computable in $e$. The sets in $\mathcal{V}$ will serve as a subbasis of a CSC space.
    \item Close $\mathcal{V}$ under finite intersection to get the collection $\mathcal{U}$ of basic open sets and the computable $k$ function. This can be done via primitive recursion. Thus we have a CSC space $X_e = (\omega, \mathcal{U}, k)$.
    \item Note that $X_e$ is uniformly computable in $e$. Argue via the $s$-$m$-$n$ theorem. Formally, $X_e$ has index $\left\langle m_e, n_e \right\rangle$, and the map $e \mapsto \left\langle m_e, n_e\right\rangle$ is computable. This map serves as our many-one reduction within $CSC$.
    \item Prove that $e \in B$ if and only if $X_e$ has the topological property described by $A$, and that $X_e$ always has the topological property described by $C$.
\end{enumerate}

In practice, steps (3) and (4) will be implicit within the proofs.
With this setup in mind, we can now present our index set calculations for separation principles.
We begin by considering the standard principles of $T_0,T_1,T_2$ and $T_3$ and finish by considering the non-integer separation principles of $T_{\frac 12}$ and $T_{2 \frac12}$.

\subsection{Index sets for $T_0,T_1,T_2$ and $T_3$}
We begin with the most basic topological principle of $T_0.$

\begin{theorem}
    The set $T_0$-$CSC = \{e : \text{$e$ is an index of a $T_0$ CSC space}\}$ is $\Pi^0_2$-complete within $CSC$.
\end{theorem}
\begin{proof}
    The following $\Pi^0_2$ formula defines the $T_0$ spaces within $CSC$:
    \begin{equation*}
        \forall x, y (x \neq y \rightarrow \exists i ((x \in U_i \land y \notin U_i) \lor (x \notin U_i \land y \in U_i))).
    \end{equation*}
    Recall that $\mathrm{Tot} = \{e : \text{$\Phi_e$ is total}\}$ is $\Pi^0_2$-complete. Let $e \in \omega$. Define $\mathcal{V} = (V_i)_{i \in \omega}$ by
    \begin{align*}
        V_0 &= \omega\\
        V_{\left\langle x, s\right\rangle + 1} &= \begin{cases}
            \{x + 1\} & \text{if $\Phi_{e, s}(x) \downarrow$}\\
            \omega & \text{otherwise}
        \end{cases}
    \end{align*}
    for all $x, s \in \omega$. We claim the resulting CSC space $X_e$ is $T_0$ if and only if $e \in \mathrm{Tot}$.

    Suppose $e \in \mathrm{Tot}$. Let $x$ and $y$ be distinct non-zero natural numbers. Since $\Phi_e$ is total, there are $s$ and $t$ such that $\Phi_{e, s}(x-1) \downarrow$ and $\Phi_{e, t}(y-1) \downarrow$. By construction, $V_{\left\langle x-1, s \right\rangle} = \{x\}$ and $V_{\left\langle y-1, s \right\rangle} = \{y\}$, so $\{x\}$ and $\{y\}$ are open in $X_e$. Additionally, $V_{x-1, s}$ does not contain $0$. Conclude that the space $X_e$ is $T_0$.

    Suppose $e \notin \mathrm{Tot}$, and fix $x$ with $\Phi_e(x) \uparrow$. Note that the only open set containing $0$ is $\omega$. For all $s$, we have $V_{\left\langle x,s\right\rangle} = \omega$, so the only open set containing $x + 1$ is $\omega$. Deduce that the space $X_e$ is not $T_0$.

    We have thus shown that $X_e$ is $T_0$ if and only if $e \in \mathrm{Tot}$.
\end{proof}

An even simpler construction can be used to show that the index set of $T_1$ CSC spaces is $\Pi^0_2$-complete within $CSC$. The following theorem is a stronger result, since we will mandate that the space $X_e$ must always be $T_0$.

\begin{theorem}\label{thm:T1Completeness}
    The set $T_1$-$CSC = \{e : \text{$e$ is an index of a $T_1$ CSC space}\}$ is $\Pi^0_2$-complete within $T_0$-$CSC$.
\end{theorem}
\begin{proof}
    The following $\Pi^0_2$ formula defines the $T_1$ spaces within $T_0$-$CSC$:
    \begin{equation*}
        \forall x, y (x \neq y \rightarrow \exists i, j(x \in U_i \land y \in U_j \land x \notin U_j \land y \notin U_i)).
    \end{equation*}
    We again use the $\Pi^0_2$-completeness of $\mathrm{Tot}$. Let $e \in \omega$. Define $\mathcal{V} = (V_i)_{i \in \omega}$ by
    \begin{align*}
        V_{2x} &= [x, \infty) = \{z \in \omega : z \geq x\}\\
        V_{2\left\langle x, s \right\rangle + 1} &= \begin{cases}
            \{x\} & \text{if $\Phi_{e, s}(x)\downarrow$}\\
            \omega & \text{otherwise}
        \end{cases}
    \end{align*}
    for all $x, s \in \omega$. We claim the resulting CSC space $X_e$ is always $T_0$, and is $T_1$ if and only if $e \in \mathrm{Tot}$. To show that $X_e$ is $T_0$, see that if $x < y$, then $V_{2y}$ is an open set containing $y$ and not $x$.

    Suppose $e \in \mathrm{Tot}$, and let $x \neq y$. Since $\Phi_e$ is total, there are $s$ and $t$ such that $\Phi_{e, s}(x) \downarrow$ and $\Phi_{e, t}(y) \downarrow$. Then $V_{2\left\langle x, s \right\rangle + 1} = \{x\}$ and $V_{2\left\langle y, t \right\rangle + 1} = \{y\}$, so $X_e$ is $T_1$.

    Suppose $e \notin \mathrm{Tot}$, and fix $x$ with $\Phi_e(x) \uparrow$. By construction, any open set containing $x$ must contain $[x, \infty)$, so $X_e$ is not $T_1$.
\end{proof}

Unlike the $T_0$ and $T_1$ levels, Hausdorff spaces lie higher up in the arithmetic hierarchy at $\Pi_3^0$.
Our construction here will always produce a discrete space, so it also shows that the property of being discrete is $\Pi_3^0$-hard within $T_1$-$CSC$.
We will see a strengthening of the result regarding discrete spaces in the next section.

\begin{theorem}\label{thm:T2Completeness}
    Let $T_2$-$CSC$ be the index set of Hausdorff CSC spaces, and let $Disc$-$CSC$ be the index set of discrete CSC spaces. Both $T_2$-$CSC$ and $Disc$-$CSC$ are $\Pi^0_3$-complete within $T_1$-$CSC$.
\end{theorem}
\begin{proof}
    The following $\Pi^0_3$ formula defines the Hausdorff spaces within $T_1$-$CSC$:
    \begin{equation*}
        \forall x, y \exists i, j (x \in U_i \land y \in U_j \land \forall z (z \in U_i \rightarrow z \notin U_j)).
    \end{equation*}
    The following $\Pi^0_3$ formula defines the discrete spaces within $T_1$-$CSC$:
    \begin{equation*}
        \forall x \exists i \forall y (y \in U_i \leftrightarrow x = y).
    \end{equation*}
    Recall that the set $\mathrm{CoInf} = \{e : \text{$W_e$ is coinfinite}\}$ is $\Pi^0_3$-complete. Let $e \in \omega$. Define $\mathcal{V} = (V_i)_{i \in \omega}$ by
    \begin{align*}
        V_{2\left\langle x, y \right\rangle} &= \{x\} \cup [y, \infty)\\
        V_{2\left\langle x, y \right\rangle + 1} &= \begin{cases}
            \{x\} \cup \{s : \Phi_{e, s}(y) \downarrow=1\} & \text{if $x \leq y$}\\
            \omega & \text{otherwise}
        \end{cases}
    \end{align*}
    for all $x, y \in \omega$. We claim the resulting CSC space $X_e$ is always $T_1$, discrete if $e \in \mathrm{CoInf}$, and not Hausdorff if $e \notin \mathrm{CoInf}$.

    Let $x < y$. Then $V_{2\left\langle x, y + 1 \right\rangle}$ is an open set containing $x$ and not $y$, and $V_{2\left\langle y, y \right\rangle}$ is an open set containing $y$ and not $x$. Thus $X_e$ is $T_1$.

    Suppose $e \in \mathrm{CoInf}$. We show $X_e$ is discrete and hence also Hausdorff. Let $x \in \omega$. Since $W_e$ is coinfinite, there is $y \geq x$ such that $\Phi_e(y) \uparrow$ or $\Phi_e(y) \downarrow=0$. It follows that $V_{2\left\langle x, y \right\rangle + 1} = \{x\}$. Thus $\{x\}$ is open for all $x \in \omega$, so $X_e$ is discrete and Hausdorff.

    Suppose $e \notin \mathrm{CoInf}$. We show $X_e$ is not Hausdorff and hence also not discrete. Since $W_e$ is cofinite, we can fix $x \in \omega$ such that $\Phi_e(y) \downarrow=1$ for all $y \geq x$. Then for all $y \geq x$, there is $s$ such that $\Phi_{e, s}(y) \downarrow$, so $V_{2\left\langle x, y\right\rangle + 1}$ is cofinite. Thus, every open set containing $x$ is cofinite. The same argument applies to $x + 1$; every open set containing $x + 1$ is cofinite. It follows that there do not exist disjoint open sets separating $x$ and $x + 1$, so $X_e$ is neither Hausdorff nor discrete.
\end{proof}

We now consider the regular CSC spaces.
As mentioned previously, this set sits at the $\Pi_5^0$ level, so it requires more detailed consideration than the arguments up to this point. 
The upper bound of $\Pi_5^0$ does not come immediately from the basic definition of regular spaces.
This definition usually quantifies over the closed sets of the topological space.
Note that this is a real number (second-order) quantifier as a general closed set $C$ is coded by the countable set of opens disjoint from $C$.
The way to get around this issue is to find an equivalent definition of regularity that is stated fully in terms of the basic open sets.
This is achieved by the following lemma. 

\begin{lemma}
    A CSC space $(X, \mathcal{U}, k)$ is regular if and only if for all $x \in X$ and $i \in \omega$ with $x \in U_i$, there is $j \in \omega$ such that $x \in U_j$ and $\overline{U_j} \subseteq U_i$, where $\overline{U_j}$ denotes the closure of $U_j$.
\end{lemma}
\begin{proof}
    Suppose that for all $x \in X$ and $i \in \omega$ with $x \in U_i$, there is $j \in \omega$ such that $x \in U_j$ and $\overline{U_j} \subseteq U_i$. Let $x \in X$ and let $F$ be a closed set in $X$ with $x \notin F$. Then $X \setminus F$ is open, so pick $i \in \omega$ such that $x \in U_i \subseteq X \setminus F$. By assumption, there is $j \in \omega$ with $x \in U_j$ and $\overline{U_j} \subseteq U_i$.

    Let $U = U_j$ and $V = X \setminus \overline{U_j}$. Then $U$ and $V$ are disjoint open sets, and $x \in U$. We also have $F \subseteq F \setminus U_i \subseteq X \setminus U_i \subseteq X \setminus \overline{U_j} = V$, so $F \subseteq V$. Hence $X$ is regular.

    Conversely, suppose $X$ is a regular CSC space, and let $x \in U_i$ for some $x \in X$ and $i \in \omega$. Let $F = X \setminus U_i$. By regularity, there are disjoint open sets $U$ and $V$ with $x \in U$ and $F \subseteq V$. Since $U$ and $V$ are disjoint and $X \setminus U_i \subseteq V$, it must be that $U \subseteq U_i$. Fix $j$ such that $x \in U_j \subseteq U$. We show $\overline{U_j} \subseteq U_i$.

    Suppose there is $y \in \overline{U_j} \setminus U_i$. Then $y \notin U_i$, so $y \in F \subseteq V$. As $U$ and $V$ are disjoint, this means $y \notin U$. Since $U_j \subseteq U_i$, it must be that $y$ is a limit point of $U_j$, meaning every open set containing $y$ contains an element of $U_j$ distinct from $y$. This is a contradiction; $V$ is an open set containing $y$, but $V$ is disjoint from $U$ and thus cannot contain an element of $U_j$. Therefore $\overline{U_j} \subseteq U_i$ as required.
\end{proof}

Direct translation of the definition implied by this lemma gives out desired upper bound for the complexity of regularity.

\begin{proposition}
The $T_3$ spaces are $\Pi_5^0$.
\end{proposition}

\begin{proof}
By the above lemma, saying a space is regular is equivalent to saying that given a point $x$ in a basic open $U_i$, there is a basic open $U_j$ such that $x\in U_j$ and $\overline{U_j}\subseteq U_i$.
In other words, the following sentence characterizes the $T_3$ spaces
\[\forall x,i \exists j ~ x\in U_j \land \forall y (\forall \ell ~ y\in U_\ell \to \exists z ~ z\in U_\ell \land  z\in U_j) \to y\in U_i. \]
\end{proof}

We now prove the desired hardness result, demonstrating that this definition is best possible.

\begin{theorem}\label{thm:T3Completeness}
    Let $T_3$-$CSC$ be the set of regular $CSC$ spaces.
    $T_3$-$CSC$ is $\Pi_5^0$ complete within within $T_2$-$CSC$.
\end{theorem}

\begin{proof}
    Our construction is built out of several basic modules of increasing complexity.
    Given $e$ coding a c.e.\ set $W_e$, we construct the following CSC space.
    We follow the convention that $\Phi_e$ enumerates at most one element into $W_e$ at each stage.
    Let $W_{e,s}^*$ be the computable set of stages at least $s$ where $W_e$ gets a new element.
    We construct the topology $\mathcal{V}_e=(V_i)_{i\in\omega}$ on $W_{e,0}^*\cup\{0\}$ where
    \[x\in V_{2s} \iff x\in W_{e,s}^*\cup\{0\} \text{   and   } V_{2s+1}=W_{e,s+1}^*-W_{e,s}^*.\]

    \begin{claim}\label{claim:basicUnit}
        The element represented by $0$ in $\mathcal{V}_e$ is isolated if and only if $W_e$ is finite.
    \end{claim}

    \begin{proof}
    If $W_e$ is finite, then there is some stage $s$ where $W_{e,s}^*$ is empty.
    In particular, $W_{e,s}^*\cup\{0\}=\{0\}$ isolates the point $0$.

    If $W_e$ is infinite, then every open set around $0$ contains a subset of the form $W_{e,s}^*\cup\{0\}$.
    Furthermore, each of the $W_{e,s}^*$ is infinite.
    If one were finite, observe that $|W_e|=|W_{e,s}\cup W^*_{e,s}|$, and the right-hand side would be finite, a contradiction to the assumption that $W_e$ is infinite.
    This means that $0$ is not isolated; in other words, it is a limit point, as every open set around $0$ contains another element.
    \end{proof}

    It is worth explicitly noting that, by construction, all of the points that are not $0$ are always isolated in $\mathcal{V}_e$.

    Geometrically, $\mathcal{V}_e$ can be thought of as the point $0$ being placed at the left end point of $[0,1]$ while the $n^{th}$ point enumerated into $W_e$ adds in the element $\frac{1}{2^n}$.
    While $0$ remains isolated after only finitely many of these actions, it becomes a limit point after infinitely many.

    We now construct a larger basic unit of the construction.
    This unit will be made up of countably many $\mathcal{V}_e$ topologies and an additional point that we call $x^*$.
    The topology will be regular at $x^*$ ---  meaning that if $x^*\in U_i$, there is a $U_j$ with $x^*\in U_j$ and $\overline{U_j}\subseteq U_i$ --- if and only if a $\Sigma_4^0$ condition is met.

    Let $A$ be a $\Sigma^0_4$ set and $f(a,c)$ be a computable function such that
    \[a\in A\iff \exists b\forall c\geq b ~ W_{f(a,c)} \text{ is finite.}\]
    Fix $a \in \omega$. Let $\mathcal{U}_{a,A}$ be the topology consisting of the basis for the topologies $\mathcal{V}_{f(a,c)}$ for each $c\in\omega$ and the following additional open sets:
    \[U_i=\{x^*\}\cup \bigcup_{j\geq i}\big(\mathcal{V}_{f(a,j)}-0_{j}\big).\]
    In the above, the notation $0_j$ denotes the element $0$ in the copy of $\mathcal{V}_{f(a,j)}$.
    In other words, $\mathcal{V}_{f(a,j)}-0_j$ is the space $\mathcal{V}_{f(a,j)}$ without the element that may or may not be a limit point of the other elements in $\mathcal{V}_{f(a,j)}$.
    In short, $\mathcal{U}_{a,A}$ is the disjoint union of the $\mathcal{V}_{f(a,c)}$ along with an extra point $x^*$ that is ``at the limit'' of the $\mathcal{V}_{f(a,c)}-0_c$.

    \begin{claim}\label{claim:biggerUnit}
    $\mathcal{U}_{a,A}$ is regular at $x^*$ if and only if $a\in A$.

    \end{claim}

    \begin{proof}
    Say that $a\in A$.
    Let $b'$ be the witness showing that $\forall c\geq b'~W_{f(a,c)} \text{ is finite.}$
    Let $U_i$ be a basic open set containing $x^*$.
    We find a basic open set containing $x^*$ whose closure is in $U_i$.
    Let $i'=\max(i,b')$.
    We claim that $\overline{U_{i'}}=U_{i'}$.
    This is enough to witness the desired properties as $x^*\in U_{i'}\subseteq U_i$ by construction.
    Let $x\not\in U_{i'}$. We must show that $x\not\in\overline{U_{i'}}$, or equivalently, we must find an open set $V_x$ with $x\in V_x$ and $V_x\cap U_{i'}=\emptyset$.
    Say that $x\in \mathcal{V}_{f(a,c)}$ with $c<i'$.
    The set $\mathcal{V}_{f(a,c)}$ is then itself such a $V_x$.
    On the other hand, say that $x\in \mathcal{V}_{f(a,c)}$ with $c\geq i'$.
    Because $x\not\in U_{i'}$, this means that $x=0_c$ with $c\geq i'$.
    Because $W_{f(a,c)}$ is finite, this gives that, by Claim \ref{claim:basicUnit}, $0_c$ is isolated in $\mathcal{V}_{f(a,c)}$ and therefore is also isolated in $\mathcal{U}_{a,A}$.
    This means that $\{0_c\}=V_x$ is an acceptable witness.

    Say that $a\not\in A$.
    We claim that no open set $X\subseteq U_0$ containing $x^*$ has $\overline{X}\subseteq U_0$.
    Any such $X$ must contain some $U_i$.
    Furthermore, as $X\subseteq U_0$, we have $0_j \notin X$ for all $j$.
    That said, because $a\not\in A$ there is a $c\geq i$ such that $W_{f(a,c)}$ is infinite, and therefore $0_c$ is a limit point of the other points in $\mathcal{V}_{f(a,c)}$ by Claim \ref{claim:basicUnit}.
    This means that $0_c\in \overline{\mathcal{V}_{f(a,c)}-0_c} \subseteq \overline{X}$, yet $0_c\not\in X$, as desired.
    \end{proof}

    We now complete the construction of the desired space.
    Given a $\Pi_5^0$ set $B$ write
    \[n\in B \iff \forall m ~ n\in A_m,\]
    where $A_m$ is a $\Sigma_4$ condition.
    Let $\mathcal{Y}_n$ be the disjoint union of the spaces $\mathcal{U}_{n,A_m}$ over all values of $m$.
    We use $x^*_m$ to denote the element referred to in Claim \ref{claim:biggerUnit} for the subspace $\mathcal{U}_{n,A_m}$.
    We use $0_{c,m}$ to denote the $c^{th}$ potential limit point in $\mathcal{U}_{n,A_m}$.
    In general, a subscript $m$ is added to the above-established notation to distinguish the different copies of $\mathcal{U}_{n,A_m}$.
    The reduction indicated in the theorem is given by the following claims.

    \begin{claim}\label{claim:T2}
        $\mathcal{Y}_n$ is $T_2$.
    \end{claim}

    \begin{proof}
        Consider two elements $x,y\in\mathcal{Y}_n$.
        The only (potentially) non-isolated points are the $x^*_m$ and the $0_{c,m}$.
        Any two isolated points are separated by the open sets only containing themselves.

        $x^*_m$ and $x^*_{m'}$ for $m\neq m'$ are separated by the opens $\mathcal{U}_{n,A_m}$ and $\mathcal{U}_{n,A_{m'}}$.
        This also separates $x^*_m$ from any other element in $\mathcal{U}_{n,A_{m'}}$.
        $x^*_m$ is separated from $0_{c,m}$ via the open sets $U_{c+1,m}$ containing $x^*_m$ and $\mathcal{V}_{f(a,c,m)}$ containing $0_{c,m}$.
        This also separates $x^*_m$ from any other element in $\mathcal{V}_{f(a,c,m)}$.

        $0_{c,m}$ is similarly separated from other points not in the same copy of $\mathcal{V}_{f(a,c,m)}$.
        Lastly, $0_{c,m}$ is separated by other (isolated) points in $\mathcal{V}_{f(a,c,m)}$ by considering $V_{2s,c,m}$ where $s$ is larger than the index of the isolated point being considered.

        This covers all of the possible cases for pairs of points in $\mathcal{Y}_n$, showing that it is always Hausdorff. 
    \end{proof}

    \begin{claim}\label{claim:T3Reduction}
        $\mathcal{Y}_n$ is $T_3$ if and only if $n\in B$.
    \end{claim}

    \begin{proof}
        If $n\not\in B$ then for some $m$, $n\not\in A_m$.
        In particular, $\mathcal{U}_{n,A_m}$ fails to be regular at $x^*_m$ by Claim \ref{claim:biggerUnit}.
        This means that $\mathcal{U}_{n,A_m}$, and indeed $\mathcal{Y}_n$, cannot be $T_3$.

        If $n\in B$, then for each $m$, $n\in A_m$.
        By Claim \ref{claim:biggerUnit}, $\mathcal{U}_{n,A_m}$ is regular at every $x^*_m$.
        Because checking regularity is a local condition, $\mathcal{Y}_n$ is regular at every $x^*_m$.
        By Claim \ref{claim:T2}, every element in $\mathcal{Y}_n$ is closed as a subset of $\mathcal{Y}_n$.
        Therefore, $\mathcal{Y}_n$ is regular at every isolated point.
        (Every isolated point is both open and closed, meaning that the set containing only the element is the desired witness to regularity.)
        This leaves only $0_{c,m}$ left to check.
        As every open set containing $0_{c,m}$ must contain some $V_{2s,c,m}$, it is enough to show that $V_{2s,c,m}$ is clopen.
        This is true because every non-isolated point outside of $V_{2s,c,m}$ is separated from $V_{2s,c,m}$ by $\mathcal{V}_{c,m}$.
        This means that $V_{2s,c,m}$ can always be used as a witness to regularity at $0_{c,m}$, as desired.
    \end{proof}

\end{proof}

The Uryshon metrization theorem states that all second countable $T_3$ spaces are metrizable.
This yields the following immediate corollary
\begin{corollary}
     Let $T_4$-$CSC$ be the set of normal $CSC$ spaces.
     Let $Met$-$CSC$ be the set of metrizable $CSC$ spaces.
    $T_3$-$CSC=T_4$-$CSC=Met$-$CSC$.
    In particular, both $T_4$-$CSC$ and $Met$-$CSC$ are $\Pi_5^0$ complete within $T_2$-$CSC$.
\end{corollary}

Of course, the same result applies for $T_{3\frac12}$, $T_5$, $T_6$, and any other notion that is generally intermediate between $T_3$ and metrizability.

\subsection{The notions $T_{2\frac12}$ and $T_{\frac12}$}

    We begin by studying an important intermediate notion between $T_2$ and $T_3$.

    \begin{definition}
        A topological space is $T_{2\frac12}$ or Urysohn if any two points are separated by closed neighborhoods.
        In other terms,
        \[\forall x,y \exists i,j ~ x\in U_i\land y\in U_j \land \overline{U_i}\cap\overline{U_j}=\emptyset.\]
    \end{definition}

    Given this notion, it is worth pointing out that the construction in Theorem \ref{thm:T3Completeness} also produces Urysohn spaces regardless of the input.

    \begin{corollary}
    Let $T_3$-$CSC$ be the set of Regular $CSC$ spaces and $T_{2\frac12}$-$CSC$ be the set of Urysohn $CSC$ spaces.
    $T_3$-$CSC$ is $\Pi_5^0$ complete within within $T_{2\frac12}$-$CSC$.
    \end{corollary}

    \begin{proof}
        The construction is precisely the same as the one in Theorem \ref{thm:T3Completeness}.
        The only change is in Claim \ref{claim:T2} where we must also check that the closure of the separating open sets also do not intersect.
        Going case by case, we must check that
        \begin{enumerate}
            \item $\overline{\mathcal{U}_{A_m,n}}\cap\overline{\mathcal{U}_{A_{m'},n}}=\emptyset$ for $m\neq m'$
            \item $\overline{U_{c+1,m}}\cap \overline{\mathcal{V}_{f(a,c,m)}}=\emptyset$
            \item $x\not\in \overline{V_{2s,c,m}}$ if $x\not\in V_{2s,c,m}$
        \end{enumerate}

        The first item follows from the fact that for any $m$, by construction, $\overline{\mathcal{U}_{A_m,n}}=\mathcal{U}_{A_m,n}$.
        The second item follows from the fact that, by construction,  $\overline{\mathcal{V}_{f(a,c,m)}}=\mathcal{V}_{f(a,c,m)}$ and that every point in $\mathcal{V}_{f(a,c,m)}$ is then separated from $U_{c+1,m}$ by $\mathcal{V}_{f(a,c,m)}$ itself.
        The last item follows from the analysis in Claim \ref{claim:T3Reduction} showing that $\overline{V_{2s,c,m}}=V_{2s,c,m}$.
    \end{proof}

    We now give an optimal description of the $T_{2\frac12}$ spaces as we have done for the other separation axioms.

    \begin{proposition}
        The $T_{2\frac12}$ spaces are $\Pi_5^0$.
    \end{proposition}

    \begin{proof}
        Consider the given definition
        \[\forall x,y \exists i,j ~ x\in U_i\land y\in U_j \land \overline{U_i}\cap\overline{U_j}=\emptyset.\]
        To prove the claim, we must show that $\overline{U_i}\cap\overline{U_j}=\emptyset$ is a $\Pi_3^0$ condition.
        This is the same as saying that any element is either separated from $U_i$ or $U_j$ by an open set.
        In other words,
        \[\forall z \exists k \forall w~ z\in U_k\land \Big((w\not\in U_k\lor w\not\in U_i) \lor (w\not\in U_k\lor w\not\in U_j)\Big).\]
        The above condition is $\Pi_3^0$, which completes the proof of the claim.
    \end{proof}

    \begin{theorem}
        $T_{2\frac12}$-$CSC$ is $\Pi_5^0$ complete within within $T_2$-$CSC$.
    \end{theorem}

    \begin{proof}
        There are similarities with the proof techniques used in the proof of this theorem and those used for Theorem \ref{thm:T3Completeness}.
        That said, they are different enough that they require their own exposition.

         Our construction is built out of several basic modules of increasing complexity.
         Given $e$ coding a c.e.\ set $W_e$, we construct the following CSC space.
        We follow the convention that $\Phi_e$ enumerates at most one element into $W_e$ at each stage.
        Let $W_{e,s}^*$ be the computable set of stages at least $s$ where $W_e$ gets a new element.
        Let $W=\{2x,2x+1 \vert x\in W_{e,0}^*\}$.
        We construct the topology $\mathcal{V}_e=(V_i)_{i\in\omega}$ on $W\cup\{0\}$ where
    \[2x,2x+1\in V_{3s} \iff x\in W_{e,s}^*\cup\{0\}\]\[V_{3s+1}=\{2x\vert x\in W_{e,s+1}^*-W_{e,s}^*\}\]
    \[V_{3s+2}=\{2x+1\vert x\in W_{e,s+1}^*-W_{e,s}^*\}.\]
        It is useful notation to split $\mathcal{V}_e$ into three parts: $0$, $\mathcal{V}_e^0$ and $\mathcal{V}_e^1$.
        $0$ is just that; the element $0$.
        $\mathcal{V}_e^0$ consists of all of the non-zero, even elements of $\mathcal{V}_e$.
        $\mathcal{V}_e^1$ consists of all of the odd elements of $\mathcal{V}_e$.

    \begin{claim}\label{claim:basicUnit1/2}
        The element represented by $0$ in $\mathcal{V}_e$ is in $\overline{\mathcal{V}_e^0}$ and $\overline{\mathcal{V}_e^1}$ if and only if $W_e$ is finite.
    \end{claim}

    \begin{proof}
    If $W_e$ is finite, then there is some stage $s$ where $W_{e,s}^*$ is empty.
    In particular, $W_{e,s}^*\cup\{0\}=\{0\}$ isolates the point $0$.
    This means that $0$ is not in the closure of any open set that does not contain $0$, so in particular it is not in $\overline{\mathcal{V}_e^0}$ or $\overline{\mathcal{V}_e^1}$.

    If $W_e$ is infinite, then every open set around $0$ contains a subset of the form $V_{3s}$.
    Furthermore, each of the $V_{3s}$ contains infinitely many elements from both $\mathcal{V}_e^0$ and $\mathcal{V}_e^1$.
    If one were finite, observe that $|W_e|=|W_{e,s}\cup \mathcal{V}_e^i|$, and the right hand side would be finite, a contradiction to the assumption that $W_e$ is infinite.
    This means that every open set around $0$ contains elements from both $\mathcal{V}_e^0$ and $\mathcal{V}_e^1$.
    Or, what is the same, $0$ is in $\overline{\mathcal{V}_e^0}$ and $\overline{\mathcal{V}_e^1}$.
    \end{proof}

        It is worth explicitly noting that, by construction, all of the points that are not $0$ are always isolated in $\mathcal{V}_e$.

    Geometrically, $\mathcal{V}_e$ can be thought of as the point $0$ being placed at the center of $[-1,1]$ while the $n^{th}$ point enumerated into $W_e$ adds in the elements $\frac{1}{2^n}$ and $-\frac{1}{2^n}$.
    While $0$ is not a limit point from either direction after finitely many of these actions, it becomes a limit point from both directions after infinitely many.

    We now construct a larger basic unit of the construction.
    This unit will be made up of countably many $\mathcal{V}_e$ topologies and two additional points that we call $x^*$ and $y^*$.
    There will be closed neighborhoods separating $x^*$ and $y^*$, if and only if a $\Sigma_4^0$ condition is met.

    Let $A$ be a $\Sigma_4$ set and $f(a,c)$ be a function such that
    \[a\in A\iff \exists b\forall c\geq b ~ W_{f(a,c)} \text{ is finite.}\]
    Let $\mathcal{U}_{a,A}$ be the topology consisting of the basis for the topologies $\mathcal{V}_{f(a,c)}$ for each $c\in\omega$ and the the following additional open sets
    \[U_{2i}=\{x^*\}\cup \bigcup_{j\geq i}\mathcal{V}^0_{f(a,j)},\]
    \[U_{2i+1}=\{y^*\}\cup \bigcup_{j\geq i}\mathcal{V}^1_{f(a,j)}.\]
    In short, $\mathcal{U}_{a,A}$ is the disjoint union of the $\mathcal{V}_{f(a,c)}$ along with two extra points that are ``at the limit'' of the $\mathcal{V}^0_{f(a,c)}$ and $\mathcal{V}^1_{f(a,c)}$ respectively.

    \begin{claim}\label{claim:biggerUnit1/2}
        There are closed neighborhoods separating $x^*$ and $y^*$ if and only if $a\in A$.
    \end{claim}

    \begin{proof}
    Say that $a\in A$.
    Let $b'$ be the witness showing that $\forall c\geq b'~W_{f(a,c)} \text{ is finite.}$
    Consider $x^*\in U_{2b'}$ and $y^*\in U_{2b'+1}$.
    We claim that $\overline{U_{2b'}}=U_{2b'}$ and $\overline{U_{2b'+1}}=U_{2b'+1}$.
    We show only $\overline{U_{2b'}}=U_{2b'}$ as the other case is symmetric.
    Let $x\not\in U_{2b'}$, we must show that $x\not\in\overline{U_{2b'}}$, or what is the same, find an open set $V_x$ with $x\in V_x$ and $V_x\cap U_{2b'}=\emptyset$.
    Say that $x\in \mathcal{V}_{f(a,c)}$ with $c<b'$.
    The set $\mathcal{V}_{f(a,c)}$ is then itself such a $V_x$.
    On the other hand, say that $x\in \mathcal{V}_{f(a,c)}$ with $c\geq b'$.
    Because $x\not\in U_{2b'}$, this means that $x=0_c$ or $x\in \mathcal{V}_{f(a,c)}^1$ with $c\geq b'$.
    Because $W_{f(a,c)}$ is finite, this gives that, by Claim \ref{claim:basicUnit1/2}, $0_c$ is not in $\overline{U_{2b'}}$ and is isolated.
    If $x\in \mathcal{V}_{f(a,c)}^1$, the open set $U_{2b'+1}$ separates $x$ from $U_{2b'}$ as desired.
    As this covers all cases, this gives that $\overline{U_{2b'}}=U_{2b'}$ and symmetrically $\overline{U_{2b'+1}}=U_{2b'+1}$.
    As $x^*\in U_{2b'}$ yet $x^*\not\in U_{2b'+1}$ while $y^*\in U_{2b'+1}$ yet $y^*\not\in U_{2b'}$, this gives the desired separating closed neighborhoods.

    Say that $a\not\in A$.
    Consider open sets $X$ containing $x^*$ and $Y$ containing $y^*$.
    Any such $X$ must contain some $U_{2i}$ and $Y$ must contain some $U_{2j+1}$.
    We argue that $\overline{U_{2i}}\cap \overline{U_{2j+1}}\neq \emptyset$, which is enough to demonstrate the claim.
    Because $a\not\in A$ there is a $c\geq i,j$ such that $W_{f(a,c)}$ is infinite, and therefore $0_c$ is a limit point of both $\mathcal{V}_{f(a,c)}^0$ and $\mathcal{V}_{f(a,c)}^1$ by Claim \ref{claim:basicUnit1/2}.
    This means that $0_c\in \overline{\mathcal{V}_{f(a,c)}^0}\cap\overline{\mathcal{V}_{f(a,c)}^0}\subseteq\overline{U_{2i}}\cap \overline{U_{2j+1}},$ as desired.
    \end{proof}

    We now complete the construction of the desired space.
    Given a $\Pi_5^0$ set $B$ write
    \[n\in B \iff \forall m ~ n\in A_m,\]
    where $A_m$ is a $\Sigma_4$ condition.
    Let $\mathcal{Y}_n$ be the disjoint union of the spaces $\mathcal{U}_{A_m,n}$ over all values of $m$.
    We use $x^*_m$ and $y^*_m$ to denote the elements referred to in Claim \ref{claim:biggerUnit1/2} for the subspace $\mathcal{U}_{A_m,n}$.
    We use $0_{c,m}$ to denote the $c^{th}$ potential limit point in $\mathcal{U}_{A_m,n}$.
    In general, a subscript $m$ is added to the above-established notation to distinguish the different copies of $\mathcal{U}_{A_m,n}$.
    The reduction indicated in the theorem is given by the following claims.

    \begin{claim}\label{claim:T2for1/2}
        $\mathcal{Y}_n$ is $T_2$.
    \end{claim}

    \begin{proof}
        Consider two elements $x,y\in\mathcal{Y}_n$.
        The only (potentially) non-isolated points are the $x^*_m$, $y^*_m$, and the $0_{c,m}$.
        Any two isolated points are separated by the open sets only containing themselves.

        The opens $\mathcal{U}_{A_m,n}$ and $\mathcal{U}_{A_{m'},n}$ are disjoint when $m\neq m'$, so separate any elements coming from different pieces of the disjoint union.
        $x^*_m$ is separated from $0_{c,m}$ via the open sets $U_{2c+2,m}$ containing $x^*_m$ and $\mathcal{V}_{f(a,c,m)}$ containing $0_{c,m}$.
        $y^*_m$ is separated from $0_{c,m}$ via the open sets $U_{2c+3,m}$ containing $y^*_m$ and $\mathcal{V}_{f(a,c,m)}$ containing $0_{c,m}$.
        Similarly, $x^*_m$ and $y^*_m$ are separated from any other elements in the $\mathcal{V}_{f(a,c,m)}$.
        $x^*_m$ and $y^*_m$ are separated by the disjoint open sets $U_{0,m}$ and $U_{1,m}$.

        $0_{c,m}$ is similarly separated from other points not in the same copy of $\mathcal{V}_{f(a,c,m)}$.
        Lastly, $0_{c,m}$ is separated from other (isolated) points in $\mathcal{V}_{f(a,c,m)}$ by considering $V_{3s,c,m}$ where $s$ is larger than the index of the isolated point being considered.

        This covers all of the possible cases for pairs of points in $\mathcal{Y}_n$, showing that it is always Hausdorff. 
    \end{proof}

    \begin{claim}\label{claim:T21/2Reduction}
        $\mathcal{Y}_n$ is $T_{2\frac12}$ if and only if $n\in B$.
    \end{claim}

    \begin{proof}
        If $n\not\in B$ then for some $m$, $n\not\in A_m$.
        In particular, $x^*_m$ and $y^*_m$ fail to be separated by closed neighborhoods by Claim \ref{claim:biggerUnit1/2}.
        This means that $\mathcal{U}_{A_m,n}$, and indeed $\mathcal{Y}_n$, cannot be $T_{2\frac12}$.

        If $n\in B$, then for each $m$, $n\in A_m$.
        By Claim \ref{claim:biggerUnit1/2}, $x^*_m$ and $y^*_m$ are separated by closed neighborhoods for every $m$.
        By Claim \ref{claim:T2for1/2}, every element in $\mathcal{Y}_n$ is closed as a subset of $\mathcal{Y}_n$.
        This means that every isolated point is clopen.
        If $x$ is isolated, and $X$,$Y$ separate $x$ from $y$, $x$,$\overline{Y}$ are closed neighborhoods that separate $x$ from $y$.
        This is because $x\not\in \overline{Y}$ as $\{x\}$ is an open set disjoint from $Y$.
        Note that $0_{c,m}$ and $0_{c',m}$ for $c\neq c'$ are separated by the clopens $\mathcal{V}_{c,m}$ and $\mathcal{V}_{c',m}$.
        Furthermore, $0_{c,m}$ is separated from $x^*$ and $y^*$ by the clopens $\mathcal{V}_{c,m}$ and $U_{2c+2}$ or $U_{2c+3}$ respectively.
        Lastly, elements in different $A_m$ are separated by the clopens containing the entirety of the $A_m$ they are a part of.
    \end{proof}

    \end{proof}

    There is also an important intermediate notion between $T_0$ and $T_1$ that we study for CSC spaces.

    \begin{definition}
        A topological space is $T_{\frac12}$ or $T_D$ if every point is isolated in its own closure. In other terms
        \[\forall x \exists i ~ x\in U_i\land \forall y ~ (y\in\overline{x}-x)\to y\notin U_i.\]
    \end{definition}

    $T_{\frac12}$ is important because it is the maximal class for which the Cantor-Bendixson analysis makes sense.
    It is worth pointing out that the construction in Theorem \ref{thm:T1Completeness} always produces $T_{\frac12}$ spaces.

    \begin{corollary}
    Let $T_{\frac12}$-$CSC$ be the set of $T_{\frac12}$ CSC spaces.
    $T_1$-$CSC$ is $\Pi_2^0$ complete within $T_{\frac12}$-$CSC$.
    \end{corollary}

    \begin{proof}
        We need only show that the space constructed in Theorem \ref{thm:T1Completeness} is always $T_{\frac12}$.
        Observe that any $x\in X_e$ is either isolated or every open set containing $x$ contains $[x,\infty)$.
        This means that the closure of $x$ is all of the non-isolated points less than $x$.
        That said, $x$ is always separated by these points by $[x,\infty)$.
        Therefore, $X_e$ is always $T_{\frac12}$, as desired.
    \end{proof}

    We now move to calibrate the complexity of $T_{\frac12}$-$CSC$ itself. 

    \begin{proposition}
        $T_{\frac12}$-$CSC$ is a $\Pi_4^0$ set .
    \end{proposition}

    \begin{proof}
        Given the definition,
        \[\forall x \exists i ~ x\in U_i\land \forall y ~ (y\in\overline{x}-x \to y\notin U_i),\]
        it is enough to note that $y\in\overline{x}-x$ is $\Pi_1$.
        This is true because $y\in\overline{x}$ is equivalent to saying that $\forall j ~ y\in U_j \to x\in U_j$ and $y\neq x$.
    \end{proof}

    \begin{theorem}
        $T_{\frac12}$-CSC is $\Pi_4^0$-complete within $T_0$-CSC.
    \end{theorem}

    \begin{proof}
        We first show that $T_{\frac12}$-CSC is $\Sigma_3^0$-hard within $T_0$-CSC, and use this construction as a basic module in the final construction.
        Recall that $Cof=\{e\vert \{x\vert \Phi_e(x)=1\} \text{ is cofinite}\}$, is a $\Sigma_3^0$ complete set.
        We define the following topology $\mathcal{V}_e$ on $\omega+1$.
        \begin{align*}
        V_{2x} &= [x, \omega] = \{z \in \omega+1 : z \geq x\}\\
        V_{2\left\langle x, s \right\rangle + 1} &= \begin{cases}
            \{x\} & \text{if } \Phi_{e, s}(x)=1\\
            \omega+1 & \text{otherwise}
        \end{cases}
        \end{align*}
        
        Say that $e\in Cof$.
        In this case, there is some $n$ such that all $m\geq n$ have some $s_m$ where $\Phi_{e, s_m}(m)=1$.
        In particular, every element in $[m,\omega]$ is isolated by $V_{2\left\langle m, s_m \right\rangle + 1}$.
        This means that the open set $[m, \omega]$ contains only $\omega$ and isolated points.
        Therefore, $\mathcal{V}_e$ is $T_{\frac12}$ at $\omega$.
        Furthermore, $\mathcal{V}_e$ is always $T_{\frac12}$ at all $x\in\omega$.
        In particular, $x\in[x, \omega] = V_{2x}$.
        For all $y\neq x$ with $y\in[x, \omega]$, $x\notin [y,\omega]$ yet $y\in [y,\omega] = V_{2y}$ meaning that $y\notin \overline{x}$, as needed.

        Say that $e\notin Cof$.
        In this case, consider $\omega$.
        Every open set containing $\omega$ is of the form $[x, \omega]$.
        Say that $[x, \omega]$ witnesses that $\mathcal{V}_e$ is $T_{\frac12}$ at $\omega$.
        Because $e\notin Cof$, there is some $y\in [x, \omega)$ such that $\Phi_e(y) \neq 1$, so $y$ is not isolated.
        In particular, $y$ is only in open sets of the form $[z, \omega]$, and therefore, $y\in\overline{\omega}$.
        This means that $\omega+1$ is not isolated in $[x, \omega]$.
        Thus, $\mathcal{V}_e$ is not $T_{\frac12}$.

        We now upgrade the construction to show $\Pi_4^0$-completeness.
        Let $A$ be a $\Pi^0_4$ set, and $f$ be a computable function such that $a \in A$ if and only if $\forall x (f(a, x) \in Cof)$.
        Consider $\coprod_{x\in\omega} \mathcal{V}_{f(a,x)}$.
        If $a\notin A$, then some $\mathcal{V}_{f(a,x)}$ is not $T_{\frac12}$ and so $\coprod_{x\in\omega} \mathcal{V}_{f(a,x)}$ is not $T_{\frac12}$.
        If $a\in A$, then all of the $\mathcal{V}_{f(a,x)}$ are $T_{\frac12}$.
        Being $T_{\frac12}$ is a local property, so $\coprod_{x\in\omega} \mathcal{V}_{f(a,x)}$ is also $T_{\frac12}$. 
    \end{proof}

\section{Linear Orderings as CSC Spaces}

A countable linear orderings are readily transformed into a CSC space by considering the order topology.
Specifically, the open intervals of the ordering form a base for a countable, second-countable topology.
These topologies have been studied in the CSC context before, notably in \cite{Gao04,S20,G24}.
These spaces are always $T_3$ and (equivalently) metrizable.
In fact, it is even the case that all countable $T_3$ spaces can be represented in this way \cite{L62}.
For the reasons outlined above, considering linear orderings as spaces is not a productive viewpoint to understand non-metrizable spaces.
That said, it is a highly productive viewpoint for understanding metrizable countable spaces.
An advantage of this approach is that methods from computable structure theory can be adopted to better understand the behavior of linear orderings, which can then be transferred to the behavior of their corresponding topologies.
Some care is still needed, however.
Notably, the map that sends an ordering to its order topology is not injective.
Linear orderings with significantly different order-theoretic properties may end up representing the same topological space.
This means that close attention is needed to separate the topologies generated two different linear orderings.
We show in this section that the completely metrizable CSC spaces are $\Sigma_1^1$ complete inside the metrizable CSC spaces (this statement bears similarities to \cite[Corollary 12.10]{G24}).
The same technique gives a new proof that the homeomorphism relation among metrizable CSC spaces is $\Sigma_1^1$-complete (a result that also follows from \cite[Theorem 4.2]{Gao04}).
Considering linear orderings as CSC spaces also yields a variety of other interesting index set results presented at the end of the section.

Our first goal is to provide an order-theoretic characterization of the linear orderings that give rise to completely metrizable interval topologies.
It is known that the countable, completely metrizable spaces are precisely the scattered ones, or the ones whose Cantor-Bendixson derivative is eventually a point (see, e.g., \cite[Theorems 12.1 and 12.13]{G24}).
The key definition then translates the Cantor-Bendixson rank to an order-theoretic setting.

\begin{definition}
Given a countable linear ordering $L$, define $\RK: L\to \omega_1 \cup \{\infty\}$ as
\[\RK(x) = \begin{cases}
\alpha \text{    if  } \alpha+1 \text{ is the least ordinal such that } x  \text{ is not a left or right } \alpha+1-\lim\\
\infty \text{    if } x \text{ is a left or right } \alpha-\lim \text{ for all countable ordinals } \alpha.
\end{cases}
\]

We let 
\[\RK(L)=\sup_{x\in L} \RK(x).\]
\end{definition}

For a detailed recursive definition of $\beta$-limit in a linear ordering, see \cite[Chapter II.4]{MBook}. 
The following are typical, straightforward examples of calculating this rank function.
We use $\eta$ to denote the order type of the rationals and $\zeta$ to denote the order type of the integers.

\begin{example}

\begin{enumerate}
	\item $\RK(\omega^\alpha+1)=\alpha$.
	\item $\RK(\eta)=\infty$.
	\item $\RK(\zeta\cdot\eta) = 0$.
        \item $\RK(\zeta^\alpha)=0$
\end{enumerate}
\end{example}

Note that Examples (3) and (4) demonstrate that the $\RK$ rank of a linear ordering may differ vastly from classical notions of rank for a linear ordering, such as Hausdorff rank or Scott rank \cite{GR24}.

The following lemma isolates a key property of the $\RK$ function that allows it to be practically used in an order-theoretic setting.

\begin{lemma}\label{lem:RkTechnical}
Consider $x\in L$. 
\begin{enumerate}
	\item If $\RK(x)=\alpha\in\omega_1$ then there exists an open interval $I_x$ such that $x\in I_x$ yet for all $y\neq x$ in $I_x$, $\RK(y)<\alpha$.
	\item If $\RK(x)=\infty$ then there for every open interval $I_x$ such that $x\in I_x$ there is a $y\neq x$ in $I_x$ such that $\RK(x)=\infty$.
\end{enumerate}

\end{lemma}

\begin{proof}
By the definition of the $\RK$ function, if $\RK(x)=\alpha\in\omega_1$ then $x$ is not a right or left limit of points that are themselves $\alpha$-limits.
This means there is an $\ell<x$ and $r>x$ such that there are no $\alpha$ limits inside of the interval $(\ell,r)$ other than $x$.
This interval can be taken to be our claimed $I_x$

Note that a rank infinity point $x\in L$ cannot be the limit of ordinal rank points for every countable ordinal.
This is because there are only countably many points in $L$, but there are uncountably many countable ordinals.
In other words, it must be rank infinity because it is the limit of rank infinity points. 
This is the same thing as saying that every interval containing $x$ has another rank infinity point.
\end{proof}

We can use this lemma to characterize outcomes of the strong Choquet game in terms of the function $\RK$.
The \textit{strong Choquet} game on $X$ is defined as follows: 

\begin{center}
$
  \begin{array}{lccccr}
I\;&U_0\ni x_0&   &U_1\ni x_1&   &\\
 &   &   &   &   &\cdots\\
II\;&   &V_0&   &V_1& 
\end{array}   
 $
 \end{center}

Players $I$ and $II$ take turns in playing nonempty open subsets of $X$.
In the first round, Player $I$ chooses a point $x_0$ and an open set $U_0$ containing $x_0$.
Player $II$ replies with an open set $V_0$ such that $x_0\in V_0\subseteq U_0$.
More generally, on the $2n^{th}$ round, Player $I$'s selects a point $x_{n}$ and an open set $U_{n}$ such that $V_{n-1}\supseteq U_{n}\ni x_{n}$.
On the $(2n+1)^{st}$ round, Player $II$'s selects an open set $V_{n}$ such that $x_n\in V_{n}\subseteq U_n$. 
Player $II$ wins if and only if
\[\bigcap\{V_n: n\in \omega\}=\{U_n: n\in \omega\}\ne\emptyset.\]
If Player $II$ has a winning strategy to this game for a space $X$, then the space $X$ is called \textit{strong Choquet}.
Choquet introduced the game to better understand the property of Baire, but it has found many other useful applications in the intervening years (see \cite{T87} Chapter 7 for an overview).
Notable for our purposes here, among metrizable spaces, the completely metrizable ones are exactly those that are strongly Choquet \cite{C69}.
This provides a path to a purely order-theoretic demonstration that the scattered spaces are the completely metrizable countable spaces. 
Notably, no metric is explicitly described.

\begin{proposition}\label{prop:SChiffRk}
The interval topology of a given order $L$ is strong Choquet if and only if $\RK(L)<\infty$.
\end{proposition}

\begin{proof}
Let us first assume that  $\RK(L)=\alpha<\infty$ and show that the interval topology is strong Choquet.
In particular, we will describe a winning strategy for Player II in the strong Choquet game.
On each turn, Player I will play $(J,x)$ where $J$ is an open interval containing $x$.
In response, Player II will play $J\cap I_x$ where $I_x$ is defined in Lemma \ref{lem:RkTechnical}.
This is a valid play as $x\in I_x$ by definition, and $J\cap I_x\subseteq J$.
To show that this is a winning strategy, we demonstrate that the intersection of all of the plays must be non-empty by the end of the game.
Note that Player I must change their selection of point infinitely many times.
If they do not, there is some point $y$ that they play on every turn after some turn.
However, in this case, $y$ will be in the final intersection.
Next, note that every time Player I switches their point from $x$ to $y$, $\RK(y)<\RK(x)$.
This follows from the fact that $y\in I_x$ and Lemma \ref{lem:RkTechnical}.

The above analysis guarantees that if Player I changes points infinitely many times at some finite point in the game, Player I plays $x$ such that $\RK(x)=0$.
What is the same, Player I must play a point that is not a limit point, i.e., an isolated point.
In particular, Player II will play $I_x = \{x\}$.
However, this means that Player I can only play $(\{x\},x)$ for the rest of the game, a contradiction to the fact that they change points infinitely many times.
This means that Player I is forced to play the same point after a particular turn and therefore always loses the game.

We now assume that $\RK(L)=\infty$ and show that the interval topology is not strong Choquet.
In particular, we will describe a winning strategy for Player I in the strong Choquet game.
First, note that there must be a point $x\in L$ such that $\RK(x)=\infty$ (in fact, infinitely many by Lemma \ref{lem:RkTechnical}).
This is because there are only countably many points in $L$ and the cofinality of $\omega_1$ is not countable.
Fix an enumeration $\sigma:\omega\to L$.
On turn $s$, Player I will guarantee that $\sigma(s)$ is not in the intersection of the intervals played throughout the game.
We will index the Player I moves with a triple $(\ell_s,r_s,x_s)$ to mean that Player I plays the interval $(\ell_s,r_s)$ along with the point $x_s$.
We will maintain throughout our strategy that $\RK(x_s)=\infty$.
We will index the Player II moves with the pair $(a_s,b_s)$ similarly.

On turn $0$ of the game:
\begin{itemize}
	\item Player I will set $x_0$ to be a rank infinity point that is not equal to $\sigma(0)$.
	\item If $\sigma(0)<x_0$ then Player I will set $\ell_0=\sigma(0)$ and $r_0$ to be an arbitrary point above $x_0$.
	\item If $\sigma(0)>x_0$ then Player I will set $r_0=\sigma(0)$ and $\ell_0$ to be an arbitrary point below $x_0$.
\end{itemize}

On turn $s+1$ of the game: 
\begin{itemize}
	\item Player I will set $x_{s+1}$ to be $x_s$ if  $\sigma(s+1)\neq x_s$ and $x_{s+1}$ to be a distinct rank infinity point inside of $(a_s,b_s)$ if $\sigma(s+1) = x_s$. Note that this is always well defined by Lemma \ref{lem:RkTechnical}.
	\item If $\sigma(s+1)<x_{s+1}$ then Player I will set $\ell_{s+1}=\max(\sigma(s+1),a_s)$ and $r_{s+1}=b_s$.
	\item If $\sigma(s+1)>x_{s+1}$ then Player I will set $r_{s+1}=\min(\sigma(s+1),b_s)$ and $\ell_{s+1}=a_s$. 
\end{itemize}

We claim that this is a winning strategy for the game.
First, we note that this strategy is valid; in other words, it follows the rules of the game.
This is immediate from the fact that $a_s\leq \ell_{s+1} \leq x_{s+1} \leq r_{s+1} \leq b_s$. 
To show that the strategy wins, note that for each $s$, $\sigma(s)\not\in (a_s,b_s)$.
In other words $\bigcap_s (a_s,b_s) = \emptyset$ as required.
\end{proof}

We can use the above result along with classical results about metrization \cite{C69} to achieve our desired characterization.
We note an additional equivalence with representation as a maximal filter space (as seen in \cite{MS09,MS10}), as these spaces have historically been interesting to computability theorists.
One motivation for presenting this proof as we have above, instead of relying on stronger results from the classical theory, is that in \cite{MS10} Theorem 5.3, an explicit maximal filter representation for a space is constructed from a winning strategy to the strong Choquet game.
This means that the above proof can be readily adapted to give a concrete conversion from a ranked CSC representation to a maximal filter representation.
We forgo these details here to avoid sidetracking from the main purpose of the article, but note that this may be a fruitful path for further investigation, particularly in the setting of reverse mathematics.

\begin{theorem}
Given a countable linear ordering $L$, the following are equivalent:
\begin{enumerate}
	\item $\RK(L)<\infty$.
	\item The interval topology on $L$ is strong Choquet.
	\item The interval topology on $L$ is homeomorphic to a maximal filter space.
	\item The interval topology on $L$ is completely metrizable.
\end{enumerate}
\end{theorem}

\begin{proof}
The above proposition demonstrates that items (1) and (2) are equivalent.
To see that (2) and (3) are equivalent, it is enough to show that the interval topology on $L$ is $T_1$.
This follows from Theorem 5.3 from \cite{MS10}.
The interval topology is always $T_1$; given $a<b$ $L_{<b}$ separates $a$ from $b$ and $L_{>a}$ separates $b$ from $a$.
To see that (3) and (4) are equivalent, it is enough to show that the interval topology on $L$ is regular.
This follows directly from \cite{MS09}.
As the interval topology has already been noted to be $T_1$, we only need to show that it is $T_3$.
Consider $a<b<c$.
We wish to separate $a$ from $[b,c]$ with two non-intersecting open sets.
If there is a $d$ such that $a<d<b$, we can take these sets to be $L_{<d}$ and $L_{>d}$.
If there is no such $d$, then we can take these sets to be $L_{<b}$ and $L_{>a}$.
\end{proof}

An important class of linear orderings, the scattered linear orderings, always have ordinal rank.
See \cite[Chapter 5]{Ros} for a treatment of the scattered linear orderings covering all of the basic definitions and results concerning the Hausdorff rank of such an ordering.
We let $\sim$ be the finite distance equivalence relation on an ordering in the following proof.

\begin{proposition}
If $L$ is a scattered linear ordering with Hausdorff rank $\alpha$, $\RK(L)\leq\alpha$.
\end{proposition}

\begin{proof}
This follows from a transfinite induction argument.
In the base case is $\alpha=0$, the ordering must be finite, so certainly $\RK(L)=0$.
In the successor case $\alpha=\beta+1$, $L/\sim$ is Hausdorff rank $\beta$.
If $x\in L$ is a $\beta+1$ limit, then we claim that $[x]_\sim$ is a $\beta$ limit in $L/\sim$.
This will immediately demonstrate the claim.

This will follow from its own transfinite induction argument.
In particular, we show that for all finite $n$ if $x\in L$ is a $n+1$ limit then  $[x]_\sim$ is an $n$ limit in $L/\sim$, and for all infinite $\gamma$ $x\in L$ is a $\gamma$ limit if and only if $[x]_\sim$ is a $\gamma$ limit in $L/\sim$.
In this base case, there is nothing to show.
In the successor case $\gamma=\delta+1$.
If $x$ is a $\gamma$ limit, this means that it is a (without loss of generality right) limit of $\delta$ limits $y_i$.
Consider $[z]_\sim<[x]_\sim$.
Take $y_i\geq z$ and note that $[y_i]_\sim\geq [z]_\sim$.
Note that no two $\delta$ limits can be finitely far from each other if $\delta>0$.
In other words, $[y_{i+1}]_{\sim}> [y_i]_\sim\geq [z]_\sim$.
This means that $[x]_\sim$ is a limit of the $[y_i]_\sim$.
By induction, this means that $[x]_\sim$ is an $n=\delta$ limit in the finite case and a $\gamma$ limit in the infinite case.
For a limit level $\lambda$ let $\delta_i\to\lambda$ be a fundamental sequence such that $\delta_i\to \lambda$ and $y_i\to x$ with $y_i$ a $\delta_i$ limit.
Note that if $f:\omega_1\to\omega_1$ is the function that decrements finite ordinals and fixes infinite ordinals $f(\delta_i)\to \lambda$ is still a fundamental sequence.
Just as in the successor case, we know that $[y_i]_\sim < [y_{i+1}]_\sim$ and therefore for all $[z]_\sim\leq[x]_\sim$ there is $[y_{i+1}]_{\sim}> [y_i]_\sim\geq [z]_\sim$.
In other words, $[x]_\sim$ is a limit of the $[y_i]_\sim$ and $[x]_\sim$ is therefore a limit of $f(\delta_i)$ limits by induction.
Therefore, $[x]_\sim$ is still a $\lambda$ limit as desired.

We now return to the main transfinite induction.
In the $\alpha=\lambda$ limit case, $L$ can be written as an $\omega$ sum, $\omega^*$ sum, or $\zeta$ sum of linear orderings $L_i$ with smaller Hausdorff ranks.
Note that the $\RK$ of a point does not change when passing to an interval.
This means that any given $x\in L_{i+1}$ has that its rank is determined by its rank in $L_i+  L_{i+1}+ L_{i+2}$.
By induction, this linear ordering does not have any $\lambda$ limits, so neither does $L$.
This means that $\RK(L)\leq\alpha$ as desired.
\end{proof}

We extract the following corollary from the above analysis.

\begin{corollary}\label{cor:ordMet}
The interval topology is completely metrizable for all countable ordinals.
\end{corollary}

\begin{lemma}\label{lem:NotComplete}
If a linear ordering $L$ has a point  $x\in L$ that is a limit of points in its own automorphism orbit, it has rank $\infty$.
\end{lemma}

\begin{proof}
Without loss of generality, we assume that $x$ is a right limit of $y_0<y_1<\cdots$ with $y_i$ in the automorphism orbit of $x$.
For the sake of contradiction, say that $x$ has an ordinal rank $\alpha$.
This means that $x$ is an $\alpha$ limit.
In turn, each $y_i$ is an $\alpha$ limit.
However, this means that $x$ is a limit of $\alpha$ limits and has ordinal rank greater than $\alpha$.
\end{proof}

We apply this lemma to the Harrison ordering $\H$, an ordering of type $\omega_1^{ck}\cdot(1+\eta)$ with a computable copy but no computable descending sequences.

\begin{corollary}\label{cor:notHomeo}
The interval topology of the Harrison linear ordering is not completely metrizable.
In particular, it is not homeomorphic to the interval topology of any countable ordinal.
\end{corollary}

\begin{proof}
This follows from Proposition \ref{prop:SChiffRk} and the observation that the Harrison linear ordering has rank $\infty$.
The latter observation is a result of the fact that the first point in a copy of $\omega_1^{ck}$ in the non-well-founded part of the Harrison linear ordering is a limit of points automorphic to it and Lemma \ref{lem:NotComplete}.
\end{proof}

The following is a useful proposition from \cite[Chapter XI]{MBook}  (see also: \cite{FFHKMM}).



\begin{proposition}\label{prop:reduct}
Given a $\Sigma_1^1$ set $A$ there is a uniform computable procedure $\Phi_A$ with the following properties:
\begin{enumerate}
	\item If $n\in A$ then $\Phi_A(n)$ is the index of a Harrison ordering.
	\item If $n\not\in A$ then $\Phi_A(n)$ is the index of a computable ordinal.
\end{enumerate}
\end{proposition}

We are interested in index sets not for linear orderings, but for CSC spaces.
However, there is a natural computable way to move into CSC spaces using the interval topology construction.

\begin{definition}
Given an index for a computable linear ordering $L$, let $L^\tau$ be the index for the topology generated by the intervals $(a,b)$ for $a<b$ in $L$.
This topology has a computable $k$ function given by $(a,b)\cap(c,d)=(c,b)$.
\end{definition}

\begin{theorem}
The set of pairs $(e,f)$ of homeomorphic, computable CSC topologies is $\Sigma_1^1$ complete inside of the set of pairs of indices for $T_3$ CSC topologies. 
\end{theorem}

\begin{proof}
Given a $\Sigma_1^1$ set $A$, let $\Psi(n) = (\mathcal{H}^\tau,\Phi_A(n)^\tau)$. 
If $n\in A$, we have that $\Phi_A(n)\cong \mathcal{H}$, so their order topologies are certainly homeomorphic.
If $n\not\in A$, we have that $\Phi_A(n)$ is a computable ordinal.
It follows from Corollaries \ref{cor:ordMet} and \ref{cor:notHomeo} that $\mathcal{H}^\tau\not\cong\Phi_A(n)^\tau$.
\end{proof}

\begin{proposition}
The set of indices for completely metrizable CSC spaces is $\Pi_1^1$-complete inside the set of indices for $T_3$ CSC spaces.
\end{proposition}

\begin{proof}
The hardness follows immediately from Proposition \ref{prop:reduct} along with Corollaries \ref{cor:ordMet} and  \ref{cor:notHomeo}.
The observation that complete metrizability is $\Pi_1^1$ follows from noting that it is equivalent to being one point after taking sufficiently many Cantor-Bendixson derivatives, or, what is the same, taking no continuous embedding from $\mathbb{Q}^\tau$.
\end{proof}

Working within the context of linear orderings yields several other interesting results about $T_3$ CSC spaces.
This is partly true because we can use powerful tools from computable structure theory to construct linear orderings.
Chief among these tools is the following: the Pair of Structure's theorem of Ash and Knight \cite{AK90}.
Below, for each $\alpha\in\omega_1$ we let $\leq_\alpha$ be the standard, asymmetric back-and-forth relations.
Given two copies $C$ and $D$, we say the back-and-forth relations are \textit{computable up to} a computable ordinal $\beta$ if the set of tuples $\bar{c}\in C$ and $\bar{d}\in D$ with $\bar{c}\leq_\gamma\bar{d}$ are uniformly computable in $\gamma$ for each $\gamma<\beta$.

\begin{theorem}[\cite{AK90}]\label{thm:PairOfStructures}
    Say that $A$ has a computable copy $\A$ and $B$ has a computable copy $\B$ where the back-and-forth relations are computable up to $\alpha$.
    If $\A\leq_\alpha \B$ then
    \[\{(e,f)|\C_e\cong A,C_f\cong B\}\]
    is $(\Sigma_\alpha, \Pi_{\alpha})$ hard.
\end{theorem}

This theorem is exceptionally useful on many occasions.
The reason for this is that it replaces the careful coding arguments typical of a hardness result (such as those used in the previous section of this article) with more combinatorial arguments concerning back-and-forth relations.
This becomes even more useful in the context of linear orderings because a lot of work has been done to understand the back-and-forth relations between linear orderings in the field of Scott analysis (see, e.g. \cite{GR24,GHT25}).
We demonstrate the power of these techniques with the following selected calculations.

We begin with an improvement on the discrete topology index set hardness result seen in Theorem \ref{thm:T2Completeness}.

\begin{proposition}
The set of indices for discrete topologies within $T_3-CSC$ is $\Pi_3^0$ complete.
\end{proposition}

\begin{proof}
It was shown in \cite{GR24} Theorem 3.7 that $\zeta+1+\zeta \leq_3 \zeta$.
Furthermore, it is not difficult to confirm that the standard presentations of these structures have that the back-and-forth relations are computable up to 3 (this can be explicitly checked using the description of the 2 back-and-forth types given in \cite{McountingBF} Section 4.1).
$(\zeta+1+\zeta)^\tau$ has one limit point - namely the element corresponding to the "1" in the middle.
This means that it is not discrete.
On the other hand, $(\zeta)^\tau$ has no limit points and is therefore discrete.

By Theorem \ref{thm:PairOfStructures}, this gives that there is a computable procedure that produces a computable copy of $\zeta+1+\zeta$ in a $\Sigma_3$ outcome and a computable copy of $\zeta$ in a $\Pi_3^0$ outcome.
Composing this procedure with the $^\tau$ map gives the desired reduction.
\end{proof}

We now demonstrate that the Cantor-Bendixson rank (or $\RK$ rank in the order theoretic domain) of a $T_3$ CSC space is optimally defined by the recursive presentation that we gave in this section. 

\begin{proposition}
For any $\alpha\in\omega_1^{ck}$:
\begin{enumerate}
	\item The set of indices of CSC spaces of rank at least $\alpha+1$ is a $\Sigma_{2\alpha+3}^0$ complete set within $T_3-CSC$.
	\item The set of indices of CSC spaces or rank at most $\alpha$ is a $\Pi_{2\alpha+3}$ complete set within $T_3-CSC$.
	\item The set of indices of CSC spaces that are rank exactly $\alpha$ is a  $\Pi_{2\alpha+3}$ complete set within $T_3-CSC$.
\end{enumerate}

\end{proposition}

\begin{proof}
It follows at once from \cite{GR24} Theorem 3.7 and \cite{MBook} Lemma II.38 that for all $\alpha\in\omega_1^{ck},$ $\omega^\alpha\cdot(\zeta+1+\zeta) \leq_{2\alpha+3} \omega^\alpha\cdot\zeta$.
Furthermore, it is not difficult to confirm that the standard presentations of these structures have that the back-and-forth relations are computable up to $2\alpha+3$ (in fact, an explicit definition of the back-and-forth relations up to this level is given in \cite{MBook} Lemma II.38).
$(\omega^\alpha\cdot(\zeta+1+\zeta))^\tau$ has one $\alpha+1$-limit point - namely the first element in the copy of $\omega^\alpha$ corresponding to the "1" in the middle.
This means that it is rank $\alpha+1$.
On the other hand, $(\omega^\alpha\cdot\zeta)^\tau$ has no $\alpha+1$-limit points and is therefore rank $\alpha$.

By Theorem \ref{thm:PairOfStructures}, this gives that there is a computable procedure that produces a computable copy of $\omega^\alpha\cdot(\zeta+1+\zeta)$ in a $\Sigma_{2\alpha+3}$ outcome and a computable copy of $\omega^\alpha\cdot\zeta$ in a $\Pi_{2\alpha+3}$ outcome.
Composing this map with the $^\tau$ map gives the desired reduction in each case.

Writing down the explicit definitions of these concepts gives completeness. 
Saying that an element $x$ is isolated is a $\Sigma_2$ property; $\exists i \forall y ~ x\in U_i \land (y\neq x\to y\not\in U_i).$
In other words, saying that an element is a limit is a $\Pi_2$ property.
Saying that an element is an $\alpha$-limit is given by the following recursive definition.
\[ \alpha-lim(x):=\bigwwedge_{\beta<\alpha} \forall i ~ x\in U_i \to \exists y\neq x ~ y\in U_i \land \beta-lim(y).\]
A straightforward transfinite induction shows that $\alpha-lim(x)$ is a $\Sigma_{2\alpha}$ formula.
Saying that a space is at most rank $\alpha$ is the same as saying
\[\forall x ~ \lnot (\alpha+1)-lim(x),\]
which is then a $\Pi_{2\alpha+3}$ property.
Saying that a space is of rank at least $\alpha+1$ is simply the negation of this formula.
To say that a space has rank exactly $\alpha$ is just to say that it has rank at most $\alpha$ and at least $\alpha$.
It is straightforward to confirm that these are of the desired complexities.

\end{proof}

\bibliographystyle{alpha}
\bibliography{CSCbib}

\end{document}